\documentclass[reqno, 11pt]{amsart}

\makeatletter
\let\origsection=\section \def\section{\@ifstar{\origsection*}{\mysection}} 
\def\mysection{\@startsection{section}{1}\z@{.7\linespacing\@plus\linespacing}{.5\linespacing}{\normalfont\scshape\centering\S}}
\makeatother    

\usepackage{geometry}
\numberwithin{equation}{section}
\numberwithin{figure}{section}

\usepackage{xcolor}
\usepackage{hyperref}
\hypersetup{
    unicode,
    colorlinks,
    linkcolor={red!60!black},
    citecolor={green!60!black},
    urlcolor={blue!60!black}
}

\usepackage{amsfonts,amsthm,amssymb,amsmath}
\usepackage[T1]{fontenc}
\usepackage{microtype}
\usepackage{verbatim}
\usepackage{scalerel}
\usepackage{thmtools, thm-restate}%For duplicating theorem numbers.

\usepackage{enumerate,enumitem}
\usepackage{thmtools, thm-restate}%For duplicating theorem numbers.

\usepackage{mathtools}
\usepackage{subcaption}

\newtheorem{theorem}{Theorem}
\numberwithin{theorem}{section}
\newtheorem{lemma}[theorem]{Lemma}

\newtheorem{claim}[theorem]{Claim}

\theoremstyle{definition}
\newtheorem{remark}[theorem]{Remark}
\newtheorem{defn}[theorem]{Definition}

\newtheorem{prob}{Problem}

\newcommand{\script}{\mathcal}
\newcommand{\parentheses}[1]{{\left( {#1} \right)}}

\newcommand{\p}{\parentheses}

\newcommand{\Set}[1]{{\left\lbrace {#1} \right\rbrace}}

\newcommand{\cardinality}[1]{{\left\lvert {#1} \right\rvert}}

\def\set#1:#2{\Set{{#1} \colon {#2}}}
\def\downcl#1{\lceil{#1}\rceil}
\def\upcl#1{\lfloor{#1}\rfloor}

\newcommand{\N}{\mathbb{N}}
\newcommand{\sub}{\subseteq}

\renewcommand{\triangleleft}{\vartriangleleft}
\renewcommand{\leq}{\leqslant}
\renewcommand{\geq}{\geqslant}
\renewcommand{\preceq}{\preccurlyeq}
\renewcommand{\rho}{\varrho}
\renewcommand{\subset}{\subseteq}
\renewcommand{\supset}{\supseteq}
\newcommand{\nottriangleleft}{\not\kern-1pt\mathrel{\triangleleft}}

\DeclareMathOperator{\cf}{cf}
\DeclareMathOperator{\height}{height}
\newcommand{\GT}{{\dot{G}}}

\begin{document}
\title{Proof of Halin's normal spanning tree conjecture}

\author{Max Pitz}
\address{Universit\"at Hamburg, Department of Mathematics, Bundesstra\ss e 55 (Geomatikum), 20146 Hamburg, Germany}
\email{max.pitz@uni-hamburg.de}



\begin{abstract}
Halin conjectured 20 years ago that a graph has a normal spanning tree if and only if every minor of it has countable colouring number. We prove Halin's conjecture. This implies a forbidden minor characterisation for the property of having a normal spanning tree.
\end{abstract}

\maketitle

\section{Introduction}

\subsection{Halin's conjecture}

A rooted spanning tree $T$ of a graph $G$ is called \emph{normal} if the ends of any edge of $G$ are comparable in the natural tree order of $T$. Normal spanning trees are the infinite analogue of the depth-first search trees. All countable connected graphs have normal spanning trees, but uncountable graphs might not, as demonstrated by uncountable cliques. For background on normal spanning trees, see the textbook \cite[\S1.5 and \S8.2]{Bible} and Section~\ref{sec_background} below.

A graph $G$ has  \emph{countable colouring number} if there is a well-order $\leq^*$ on $V(G)$ such that every vertex of $G$ has only finitely many neighbours preceding it in $\leq^*$. Every graph with a normal spanning tree has countable colouring number: simply well-order level by level. The na\"ive converse fails, however, as witnessed by uncountable cliques with all edges subdivided.

In \cite{halin2000miscellaneous}, Halin observed %[Lemma~7.2]
that 
%as a consequence of a theorem of Jung, 
the property of having a normal spanning tree is minor-closed, i.e.\ preserved under taking connected minors. Recall that a graph $H$ is a \emph{minor} of another graph $G$, written $H \preceq G$, if to every vertex $x \in H$ we can assign a (possibly infinite) connected set $V_x \subset V(G)$, called the \emph{branch set} of $x$, so that these sets $V_x$ are disjoint for different $x$ and $G$ contains a $V_x-V_y$ edge whenever $xy$ is an edge of $H$. 

Hence, graphs with normal spanning trees have the property that 
%not only they themselves, but 
also all their minors have countable colouring number. In \cite[Conjecture~7.6]{halin2000miscellaneous} Halin conjectured a converse to this observation. %, namely that a connected graph has a normal spanning tree if and only if every minor of it has countable colouring number. 
The purpose of this paper is  to prove Halin's conjecture.

\begin{restatable}{theorem}{main}
%\begin{theorem}
\label{thm_Halin's_conj}
A connected graph has a normal spanning tree if and only if every minor of it has countable colouring number.
%\end{theorem}
\end{restatable}

\subsection{A forbidden minor characterisation for normal spanning trees} In the same paper~\cite{halin2000miscellaneous}, % where Halin posed his normal spanning tree conjecture from above, 
Halin asked for an explicit forbidden minor characterisation for the property of having a normal spanning tree \cite[Problem~7.3]{halin2000miscellaneous}. Using the recent forbidden subgraph characterisation for the property of having countable colouring number by Bowler, Carmesin, Komj\'ath and Reiher \cite{bowler2015colouring}, such a forbidden minor characterisation %as requested by Halin 
can be deduced from Theorem~\ref{thm_Halin's_conj}. 

These forbidden minors come in two structural types: First, the class of $(\lambda,\lambda^+)$\emph{-graphs}, bipartite graphs $(A,B)$ such that $\cardinality{A}=\lambda$, $\cardinality{B}=\lambda^+$ for some infinite cardinal $\lambda$, and every vertex in $B$ has infinite degree. And second, the class of $(\kappa,S)$\emph{-graphs}, graphs whose vertex set is a regular uncountable cardinal $\kappa$ such that stationary many vertices $s \in S \subseteq \kappa$ have countably many neighbours that are cofinal below~$s$.

\begin{restatable}{theorem}{maintwo}
%\begin{theorem}
\label{thm_forbiddenminorsIntro}
	A connected graph $G$ has a normal spanning tree if and only if it contains neither a $(\lambda,\lambda^+)$-graph nor a  $(\kappa,S)$-graph as a minor with countable branch sets.
%\end{theorem}
\end{restatable}

A surprising consequence of Theorem~\ref{thm_forbiddenminorsIntro} is that a graph of singular uncountable cardinality $\kappa$ has a normal spanning tree as soon as all its minors of size strictly less than $\kappa$ admit normal spanning trees. This is not the case when $\kappa$ is regular \cite[Theorem~5.1]{pitz2020new}.

That it suffices to forbid minors with countable branch sets in Theorem~\ref{thm_forbiddenminorsIntro} has an immediate application: From it, we deduce a proof of Diestel's normal spanning tree criterion from \cite{diestel2016simple}, that a graph has a normal spanning tree provided it contains no subdivision of a `fat' $K^{\aleph_0}$, a complete graph in which every edge has been replaced by uncountably many parallel edges.  

\subsection{Relation to work by Diestel and Leader}

Halin's conjecture would have followed from Diestel and Leader's proposed forbidden minor characterisation of graphs having a normal spanning tree  \cite{DiestelLeaderNST}. Unfortunately, their result is not correct, as shown by the author in %an earlier paper
 \cite{pitz2020new}. 

Diestel and Leader \cite{DiestelLeaderNST} asserted that  the forbidden minors for the property of having a normal spanning tree are $(\aleph_0,\aleph_1)$-graphs, and \emph{Aronszajn tree-graphs}, graphs whose vertex set is an %order theoretic 
Aronszajn tree $\script{T}$ %(an order tree $(\script{T},\leq)$ of size $\aleph_1$ in which all levels and branches are countable) 
such that the down-neighbourhood of any node $t \in \script{T}$ is cofinal below~$t$.

Well-ordering an Aronszajn tree-graph level by level shows that it contains an $(\omega_1,S)${-}subgraph.
% (with $S$ being the club consisting of the first element of each limit level); 
However, \cite[Theorems~3.1 and 5.1]{pitz2020new} show that there exist $(\omega_1,S)$-graphs and also larger $(\kappa,S)$-graphs that contain neither an $(\aleph_0,\aleph_1)$-graph nor an Aronszajn tree-graph as a minor. 

%Further, \cite[Theorem~5.1]{pitz2020new} shows that any list of excluded minors characterising the property of having a normal spanning trees in ZFC must contain obstructions of arbitrarily large cardinality, as it is the case with the list from Theorem~\ref{thm_forbiddenminorsIntro}.

\subsection{Organisation of this paper}
In Section~\ref{sec_background}, we recall all facts about normal trees, $T$-graphs for order-theoretic trees $(T,\leq)$, as well as stationary sets that are needed in this paper. 

Section~\ref{sec_decomp} contains a decomposition result from which we prove Theorem~\ref{thm_Halin's_conj} by a cardinal induction in Section~\ref{sec_mainproof}. When tackling Halin's conjecture, one naturally faces the question how to best exploit the assumption that no minor of $G$ has countable colouring number. Besides extending a number of ideas from \cite{bowler2015colouring} to the minor setting, the key ingredient is to take advantage of 
the countable colouring number of 
one particular minor $\GT \preceq G$ 
%with countable branch sets 
that has the structure of a $T$-graph, provided by the theory of normal tree orders %of infinite graphs 
of Brochet and Diestel \cite{brochet1994normal}. 

Section~\ref{sec_forb} contains the proof of our forbidden minor characterisation as stated in Theorem~\ref{thm_forbiddenminorsIntro}.
Finally, in Section~\ref{sec_fat} we use this forbidden minor characterisation to give a corrected proof of Diestel's normal spanning tree criterion.

\subsection{Acknowledgements}

I would like to thank Reinhard Diestel for stimulating discussions and an insight simplifying the
construction of the normal spanning trees in Theorem~\ref{thm_Halin's_conj}.

\section{Preliminaries}
\label{sec_background}

We follow the notation in \cite{Bible}. Given a subgraph $H \subseteq G$, write $N(H)$ for the set of vertices in $G -H$ with a neighbour in $H$. A \emph{tour} is a walk that may repeat vertices but not edges.

\subsection{Normal spanning trees}

If $T$ is a graph-theoretic tree with root~$r$, we write $x \le y$
for vertices $x,y\in T$ if $x$ lies on the unique $r$--$y$ path in~$T$.
A rooted tree $T \subset  G$ \emph{contains a set $U$ cofinally} if $U \subset V(T)$ and $U$ is cofinal in the tree-order $(T,\leq)$.

A rooted tree $T \subset G$ is \emph{normal (in $G$)} if the end vertices of any $T$-path in $G$ (a path in $G$ with end vertices in $T$ but all edges and inner vertices outside of $T$) are comparable in the tree order of $T$. See also \cite[\S1.5 and \S8.2]{Bible}.
If $T$ is spanning, this reduces to the requirement that the ends of any edge of $G$ are comparable in the tree order on $T$. For a normal tree $T\subset G$, the neighbourhood $N(D)$ of every component $D$ of $G-T$ forms a chain in $T$. 

A set of vertices $U \subset V(G)$ is \emph{dispersed} (in G) if every ray in $G$ can be separated from $U$ by a finite set of vertices. 
The following theorem of Jung \cite[Satz~6]{jung1969wurzelbaume} characterises vertex sets $U \subset V(G)$ that are cofinally contained in a normal tree of $G$. See also \cite{pitz2020} or \cite[Theorem~3.5]{jancarl}.

\begin{theorem}[Jung]
\label{thm_Jung}
Let $G$ be a connected graph and $r$ any vertex of $G$. A set of vertices $U \subset V(G)$ is cofinally contained in some normal tree of $G$ with root $r$ if and only if $U$ is a countable union of dispersed sets in $G$.
\end{theorem}

\subsection{Normal tree orders and $T$-graphs}

A partially ordered set $(T,\le)$ is called an \emph{order tree} if it has a unique minimal element (called the \emph{root}) and all subsets of the form $\lceil t \rceil = \lceil t \rceil_T := \set{t' \in T}:{t'\le t}$
are well-ordered. Our earlier partial ordering on the vertex set of
a rooted graph-theoretic tree is an order tree in this sense.

Let $T$ be an order tree. A~maximal chain in~$T$ is called a \emph{branch}
of~$T$; note that every branch inherits a well-ordering from~$T$. The
\emph{height} of~$T$ is the supremum of the order types of its branches. The
\emph{height} of a point $t\in T$ is the order type of~$\mathring{\lceil t \rceil}  :=
\lceil t \rceil  \setminus \{t\}$. The set $T^i$ of all points at height $i$ is
the $i$th \emph{level} of~$T$, and we
write $T^{<i} := \bigcup\set{T^j}:{j < i}$.

The intuitive interpretation of a tree order as expressing height will also
be used informally. For example, we say that $t$ is \emph{above}~$t'$
if $t > t'$, and call $\lceil X \rceil = \lceil X \rceil _T := \bigcup \set{\lceil x \rceil}:{x\in
X}$ the \emph{down-closure} of~$X\sub T$. And we say that $X$ is \emph{down-closed}, or $X$ is a \emph{rooted subtree}, if $X=\lceil X \rceil $.

If $t < t'$, we write $[t,t'] = \set{x}:{t \leq x \leq t'}$, and call this
set a (closed) \emph{interval} in~$T$. (Open and half-open intervals in~$T$
are defined analogously.) A~subset of $T$ that is an order tree under the
ordering induced by~$T$ is a \emph{subtree} of $T$ if along with any two
comparable points it contains the interval in~$T$ between them. If $t < t'$
but there is no point between $t$ and~$t'$, we call $t'$ a \emph{successor}
of~$t$ and $t$ the \emph{predecessor} of~$t'$; if $t$ is not a successor
of any point it is called a \emph{limit}.

An order tree $T$ is \emph{normal} in a graph $G$, if $V(G) = T$
and the two ends of any edge of $G$ are comparable in~$T$. We call $G$ a
\emph{$T$-graph} if $T$ is normal in $G$ and the set of lower neighbours of
any point $t$ is cofinal in $\mathring{\lceil t \rceil}$. 
An $\omega_1$-graph is a $T$-graph for the well-order $T = (\omega_1,\leq)$.
For detailed information on normal tree orders, we refer the reader to \cite{brochet1994normal}.
If $G$ is a $T$-graph,
then every interval $[t,t']$ in~$T$ (and hence every subtree of~$T$)
is connected in~$G$, because only $t$ can be a minimal element of any of
its components. For later use we note down the following standard results about $T$-graphs, and refer the reader to \cite[\S2]{brochet1994normal} for details.

\begin{lemma}
\label{lem_Tgraphproperties}
Let $(T,\leq)$ be an order tree and $G$ a $T$-graph.
\begin{enumerate}
\item \label{itemT1} For incomparable vertices $t,t'$ in $T$, the set $\downcl{t} \cap \downcl{t'}$ separates $t$ from $t'$ in $G$.
\item \label{itemT2}Every connected subgraph of $G$ has a unique $T$-minimal element.
\item \label{itemT3} If $T' \subset T$ is down-closed, the components of $G - T'$ are spanned by the sets $\upcl{t}$ for $t$ minimal in $T-T'$.
\end{enumerate}
\end{lemma}

\subsection{Stationary sets and Fodor's lemma}
We denote ordinals by $i,j,k,\ell$, and identify $i = \set{j}:{j < i}$.
Let $\ell$ be any limit ordinal. A subset $A \subset \ell$ is \emph{unbounded} if $\sup A = \ell$, and \emph{closed} if $\sup (A \cap m) = m$ implies $m \in A$ for all limits $m < \ell$. The set $A$ is a \emph{club-set} in $\ell$ if it is both closed and unbounded.
A subset $S \subset \ell$ is \emph{stationary} (in $\ell$) if $S$ meets every club-set of $\ell$. For the following standard results about stationary sets see e.g.\ \cite[\S III.6]{Kunen}.

\begin{lemma}
\label{lem_stationary}
(1) If $\kappa$ is a regular uncountable cardinal, $S \subset \kappa$ is stationary and $S = \bigcup \set{S_n}:{i \in \N}$, then some $S_n$ is stationary.

(2) \emph{[Fodor's lemma]} If $\kappa$ is a regular uncountable cardinal, $S \subset \kappa$ stationary and $f \colon S \to \kappa$ is  such that $f(s)<s$ for all $s \in S$, then there is $i< \kappa$ such that $f^{-1}(i)$ is stationary. 
\end{lemma}

\section{Decomposing graphs into subgraphs of finite adhesion}
\label{sec_decomp}

\subsection{Statement of the decomposition result} The aim of this section is to prove a decomposition lemma which allows us to prove Halin's conjecture by cardinal induction. To state the decomposition lemma, we need a definition.

\begin{defn}[Subgraphs of finite adhesion]
\label{def_Trobust}
Given a subgraph $H \subset G$, a set of vertices $A \subset V(H)$ of the form $A=N(D)$ for some component $D$ of $G - H$ is an \emph{adhesion set}.
A subgraph $H \subset G$ has \emph{finite adhesion} if all adhesion sets in $H$ are finite.
\end{defn}

\begin{remark}
\label{rem_Trobust}
An increasing $\omega$-union of subgraphs of finite adhesion may fail to have finite adhesion. An increasing $\omega_1$-union of subgraphs of finite adhesion has itself finite adhesion.
\end{remark}

Our main decomposition result reads as follows.

\begin{lemma}[Decomposition lemma]
\label{lem_decomposition2}
Let $G$ be a connected graph of uncountable size $\kappa$ with the property that all its minors have countable colouring number. Then $G$ can be written as a continuous increasing union $ \bigcup_{i < \sigma} G_i$ of infinite, ${<}\kappa$-sized connected induced subgraphs $G_i$ of finite adhesion in $G$.
\end{lemma}

To see how to obtain a normal spanning tree from this result, the reader may wish to skip to Section~\ref{sec_mainproof} immediately.

\subsection{Normal partition trees}
\label{sec_NPT}

Instead of proving Lemma~\ref{lem_decomposition2} directly, we prove a technical strengthening using the notion of normal partition trees due to Brochet and Diestel from \cite[\S4]{brochet1994normal}:

Let $G$ be a graph and $\set{V_t}:{t \in T}$ be a partition of~$V(G)$ into non-empty
sets~$V_t$. If the index set $T$ of this partition is an order tree $(T,\leq)$, we call $(T,\leq)$ a
{\it partition tree\/} for~$G$. For vertices $v\in G$, we write $t(v)$ for
the node $t\in T$ such that $v\in V_t$. Whenever we speak of a
partition tree $T$ for~$G$, we shall assume that it comes with
a fixed partition of~$V(G)$; the sets $V_t$ and the map $v\mapsto t(v)$
will then be well-defined.

If $T$ is a partition tree for~$G$, we denote by $\GT=G/T$ the graph
obtained from $G$ by contracting the sets~$V_t$ for $t\in T$. We may
then identify $T$ with the vertex set of~$\GT$; thus, two points $t,t'
\in T$ become adjacent vertices of $\GT$ if and only if $G$ contains a
$V_t$--$V_{t'}$ edge. We call $T$ a \emph{normal} partition tree
for~$G$ if the following properties hold:
\begin{enumerate}[label=(\alph*)]
 \item\label{item_NPT1} $\GT$ is a $T$-graph,
 \item\label{item_NPT2} for every $t\in T$, the set $V_t$ is connected in~$G$ (so $\GT$ is a minor of $G$),
 \item\label{item_countablesize} for every $t\in T$, we have either $|V_t| = \cf (\height (t))$ or $|V_t| = 1$.
\end{enumerate}

For a subtree $T' \subset T$, we write $G(T') := G[\bigcup \set{ V_t}:{t \in T'}]$ for the corresponding connected induced subgraph of $G$.
We use the following result by Brochet \& Diestel, see {\cite[Theorem~4.2]{brochet1994normal}}.

\begin{theorem}[Brochet \& Diestel]
\label{thm_brochetdiestel}
Every connected graph has a normal partition tree.
\end{theorem}

The intuition behind the last of the above requirements for a normal partition tree is that $T$ should approximate $G$ as best as possible, which happens if the partition sets $V_t$ are small. For the normal partition trees $T$ considered in this paper, the branch sets $V_t$ are always at most countable, see Lemma~\ref{lem_noOmega_1chains} below. Indeed, graphs $G$ all whose minors have countable colouring number cannot contain an uncountable clique minor. 

\begin{lemma}
\label{lem_noOmega_1chains}
Let $G$ be a connected graph not containing an uncountable clique minor. Then all branches of a normal partition tree $T$ for~$G$ are at most countable; in particular all branch sets $V_t$ in $G$ are at most countable. 
\end{lemma}

\begin{proof}
If $T$ contains an uncountable branch, then by \ref{item_NPT1} and \ref{item_NPT2} the minor $\GT$ of $G$ contains an $\omega_1$-graph as a subgraph. But by \cite[Proposition~3.5]{DiestelLeaderNST}, every $\omega_1$-graph contains a $K^{\omega_1}$ minor, a contradiction. In particular, $\height(t) < \omega_1$ for all $t \in T$, 
and so the second assertion follows from property \ref{item_countablesize} of normal partition trees.
\end{proof}

We can now state the decomposition result in the form we want to prove it:

\begin{lemma}[Decomposition Lemma, $T$-graph variant]
\label{lem_decomposition}
Let $G$ have uncountable size $\kappa$ with the property that all its minors have countable colouring number. Then any normal partition tree $T$ for $G$ can be written as a continuous increasing union $ \bigcup_{i < \cf(\kappa)} T_i$ of infinite, ${<}\kappa$-sized rooted subtrees $T_i$ such that all graphs $ G(T_i)$ have finite adhesion in $G$.
\end{lemma}

Indeed, it follows from Lemma~\ref{lem_noOmega_1chains} that $G_i := G(T_i)$ are as desired for Lemma~\ref{lem_decomposition2}.

\subsection{A closure lemma} The proof of Lemma~\ref{lem_decomposition} relies on a closure lemma. In it, we use that the following two types of graphs have uncountable colouring number and therefore cannot appear as minors of $G$:
\begin{enumerate}[label=(\roman*)]
%	\item An $\omega_1$-graph (cf.\ \cite[Proposition~3.5]{DiestelLeaderNST}).
	\item\label{item_barricade} A \emph{barricade}, i.e.\ a bipartite graph with bipartition $(A,B)$ such that $|A| < |B|$ and every vertex of $B$ has infinitely many neighbours in $A$, cf.\ \cite[Lemma~2.4]{bowler2015colouring}.
	\item\label{item_aroszajn} An \emph{Aronszajn tree-graph}, i.e.\ a $T$-graph for an Aronszajn tree $T$, cf.\ \cite[Theorem~7.1]{DiestelLeaderNST}.
\end{enumerate}

\begin{lemma}
\label{lem_closureTrobust}
Let $G$ have uncountable size $\kappa$ with the property that all its minors have countable colouring number, and let $T$ be a normal partition tree for $G$. Then every infinite $X \subset T$ is included in a rooted subtree $T' \subset T$ with $|X| = |T'|$ such that $G(T')$ has finite adhesion in $G$.
 \end{lemma}

\begin{proof}
For a connected subgraph $D \subset G$ write $t_D$ for the by Lemma~\ref{lem_Tgraphproperties}(\ref{itemT2}) unique $T$-minimal element of $\set{t(v)}:{v \in D}$. We recursively build a $\subseteq$-increasing sequence $\set{T_i}:{i < \omega_1}$ of
rooted subtrees of $T$ by letting $T_0 = \lceil X \rceil_T$, defining 
$$ T_{i+1} = T_i \cup \set{t_D}:{D \text{ a component of } G - G(T_i) \text{ with } |N(D) \cap G(T_i)| = \infty}$$
at successor steps, and $T_\ell = \bigcup_{i < \ell} T_i$ for limit ordinals $\ell < \omega_1$. Finally we set $T' = \bigcup_{i < \omega_1} T_i$. Clearly, $T'$ is a rooted subtree of $T$ including $X$. %We claim that $T'$ has finite adhesion in $G$ and satisfies $|T' | = |X|$.

To see that $G(T')$ has finite adhesion in $T$, suppose for a contradiction that there is a component $D$ of $G-G(T')$ with $| N(D) \cap G(T') | = \infty$. Then there is some $i_0 < \omega_1$ such that $|N(D) \cap G(T_{i_0})| = \infty$. Hence for all $i_0 \leq i < \omega_1$, the unique component $D_i$ of $G - G(T_i)$ containing $D$ also satisfies $|N(D_i) \cap G(T_i)| = \infty$. %Since the collection $\set{D_i}:{i_0 \leq i < \omega_1}$ is nested, it follows that 
Then $\set{t_{D_i}}:{i_0 \leq i < \omega_1}$ forms an uncountable chain in $T$, contradicting Lemma~\ref{lem_noOmega_1chains}.

To see that $|T'| = |X|$, observe that since $T$ contains no uncountable chains by Lemma~\ref{lem_noOmega_1chains}, %the down-closure of every vertex is countable, implying 
we have $|T_0| = |X|$. We now prove by transfinite induction on $i< \omega_1$ that $|T_i| = |X|$. The cases where $i$ is a limit are clear, so suppose $i = j+ 1$. By the induction hypothesis, $|T_j| = |X|$. If $|T_{j+1}|>|T_j|$, then the 
bipartite minor of $G$ obtained by contracting all components $D$ of $G - G(T_j)$ with $t_D \in T_{j+1} \setminus T_j$ to form the $B$-side, and all vertices in $G(T_j)$ forming the $A$-side would be a barricade $(A,B)$ by the second part of Lemma~\ref{lem_noOmega_1chains}, contradicting \ref{item_barricade}. 

If $X$ is uncountable, then $|T'| = | \bigcup_{i < \omega_1} T_i | = \aleph_1 \cdot |X| = |X|$. So suppose for a contradiction that $X$ is countable and $|T'| = \aleph_1$. Contracting the rooted subtree $T_0$ to a vertex $r$ in $T'$ gives rise to an order tree $T''$ with root $r$. Since $T_0 \subset T'$ is a rooted subtree and so $\GT[T_0]$ is connected, this contraction results in a minor $G''$ of $\GT$ which is a $T''$ graph. By construction, nodes in $T_{i} \setminus \bigcup_{j < i} T_j$ for $i \geq 1$ belong to the $i$th level of $T''$, and hence all levels of $T''$ are countable. Finally, since $T''$ like $T'$ and $T$ contains no uncountable chains, it follows that $T''$ is an Aronszajn tree. Since $G'' \preceq \GT \preceq G$, we have found an Aronszajn tree minor of $G$, contradicting \ref{item_aroszajn}.
\end{proof}

\subsection{Proof of Lemma~\ref{lem_decomposition}}
We are now ready to prove our Decomposition Lemma~\ref{lem_decomposition}. The proof will be divided into two cases depending on whether $\kappa = |G|$ is regular or singular.

\begin{proof}[Proof of Lemma~\ref{lem_decomposition} for regular uncountable $\kappa$]
Let $\dot{\leq}$ be a well-order of $V(\GT)$ witnessing that $\GT$ has countable colouring number, i.e.\ so that every vertex has only finitely many neighbours preceding it in $\dot{\leq}$. We may choose $\dot{\leq}$ to be of order type $|\GT|$, see e.g.\ \cite[Corollary~2.1]{EGJKP19}.

By the `in particular part' of Lemma~\ref{lem_noOmega_1chains}, we have $\kappa = |G|=|\GT| = |T|$. Fix an enumeration $V(T)=\set{t_i}:{i < \kappa}$. We recursively define a continuous increasing sequence $\set{T_i}:{i < \kappa}$ of rooted subtrees of $T$ with 
\begin{enumerate}
	\item $t_i \in T_{i+1}$ for all $i < \kappa$,
	\item each $G(T_i)$ has finite adhesion in $G$,
%	\item $G(T_i)$ forms a proper initial segment of $(V(G),\dot{\leq}_G)$, and
	\item the vertices of $T_i$ form a proper initial segment of $(V(\GT),\dot{\leq})$
\end{enumerate}
Let $T_0 = \emptyset$. In the successor step, suppose that $T_i$ is already defined. Let $T_i^0 := T_i \cup \lceil t_i \rceil$. At odd steps, use Lemma~\ref{lem_closureTrobust} to fix  a rooted subtree $T_i^{2n+1}$ of $T$ including $T_i^{2n}$ of the same size as $T_i^{2n}$ so that $G(T_i^{2n+1})$ has  finite adhesion in $G$. At even steps, let $T_i^{2n+2}$ be the smallest subtree of $T$ including the down-closure of $T_i^{2n+1}$ in $(V(\GT),\dot{\leq})$.
Define $T_{i+1} = \bigcup_{n \in \N}T_i^n$. By construction, $T_{i+1}$ is a rooted subtree of $T$ with $t_i \in T_{i+1}$, %. Since all $T_i^{2n}$ for $n \in \N$ are down-closed in $(V(\GT),\dot{\leq})$, it is also clear that 
and $T_{i+1}$ forms an initial segment of $(V(\GT),\dot{\leq})$. To see that this initial segment is proper, one verifies inductively that $|T_i^{n}| < \kappa$; since $\kappa$ has uncountable cofinality, this gives $|T_{i+1}| < \kappa$, too. %To see that $T_{i+1}$ has finite adhesion note first that since all  $T_i^{2n+1}$ for $n \in \N$ are down-closed in $(T,\leq)$, so is $T_{i+1}$. 

Hence, it remains to show that $G(T_{i+1})$ has finite adhesion in $G$. 
Suppose otherwise that there exists a component $D$ of $G-G(T_{i+1})$ with infinitely many neighbours in $G(T_{i+1})$. If we let $d =t_D$, then $t(N(D)) \subset \mathring{\lceil d \rceil}_T$ holds by definition of a normal partition tree. We claim that $d$ must be a limit of $T$. Indeed, for any $x <_T d$, Lemma~\ref{lem_Tgraphproperties}(\ref{itemT3}) implies that $x \in T_i^{2n+1}$ for some $n \in \N$. Since $G(T_i^{2n+1})$ has finite adhesion, it follows that $N(D) \cap G(T_i^{2n+1})$ is finite. In particular, only finitely many neighbours $v \in N(D)$ satisfy $t(v) \leq_T x$. Hence, at least one neighbour $v \in N(D)$ satisfies $x <_T t(v) <_T d$; %; as this holds for arbitrary $x<_Td$, it follows follows that 
so $d$ is a limit.

 By the definition of a $T$-graph, $d$ has infinitely many  $\GT$-neighbours below it, and hence in $T_{i+1}$. However, since $T_{i+1}$ forms an initial segment in $(V(\GT),\dot{\leq})$ not containing $d$, it follows that $d$ is preceded by infinitely many of its neighbours in $\dot{\leq}$, contradicting %that $\dot{\leq}$ witnesses that $\GT$ has countable colouring number.
 the choice of $\dot{\leq}$.

For limits $\ell < \kappa$ we define $T_\ell = \bigcup_{i < \ell} T_i$. One verifies as above that $T_\ell$ is a rooted subtree of $T$ that forms a proper initial segment in $(V(\GT),\dot{\leq})$ such that $G(T_\ell)$ has finite adhesion in $G$.
\end{proof}

\begin{proof}[Proof of Lemma~\ref{lem_decomposition} for singular uncountable $\kappa$]
This case follows \cite[\S 4]{bowler2015colouring}, where we replace the notion of \emph{robustness} from \cite[Definition~4.1]{bowler2015colouring} by our notion of having \emph{finite adhesion} from Definition~\ref{def_Trobust}, and use Lemma~\ref{lem_closureTrobust} and Remark~\ref{rem_Trobust} instead of \cite[Lemma~4.3]{bowler2015colouring} and \cite[Remark~4.2]{bowler2015colouring}. The complete argument follows for convenience of the reader.

Let us enumerate $V(T) = \set{t_i}:{i < \kappa}$ and fix a continuous increasing sequence $\set{\kappa_i}:{i < cf(\kappa)}$ of cardinals with limit $\kappa$, where $\kappa_0 > cf(\kappa)$ is uncountable.
We build a family $$\set{T_{i,j}}:{i < cf(\kappa), \; j < \omega_1}$$ of rooted subtrees of $T$ with $G(T_{i,j})$ of finite adhesion in $G$, with each $T_{i,j}$ of size $\kappa_i$. This will be done by a nested recursion on $i$ and $j$. When we come to choose $T_{i,j}$, we will already have chosen all $T_{i',j'}$ with $j' < j$, or with both $j' = j$ and $i' < i$. Whenever we have just selected such a subtree $T_{i,j}$, we fix immediately an arbitrary enumeration $\set{t^k_{i,j}}:{k < \kappa_i}$ of this tree. We impose the following conditions on this construction:
\begin{enumerate}
	\item $\set{t_k}:{k < \kappa_i} \subset T_{i,0}$ for all $i$,
	\item $\bigcup \set{T_{i',j'}}:{i' \leq i, j' \leq j} \subset T_{i,j}$ for all $i$ and $j$,
	\item $\set{t^k_{i',j}}:{k < \kappa_i} \subset T_{i,j+1}$ for all $i < i' < cf(\kappa)$ and $j$.
\end{enumerate}
These three conditions specify some collection of $\kappa_i$-many vertices which must appear in $T_{i,j}$. By Lemma~\ref{lem_closureTrobust} we can extend this collection to a rooted subtree $T_{i,j}$ of the same size such that $G(T_{i,j})$ has finite adhesion in $G$. This completes the description of our recursive construction.

Condition (3) ensures that
\begin{enumerate}
	\item[(4)] $T_{\ell,j} \subset \bigcup_{i < \ell} T_{i,j+1}$ for all limits $\ell < cf(\kappa)$ and all $j$.
\end{enumerate}
Indeed, since $\kappa_\ell = \bigcup_{i < \ell} \kappa_i $ by continuity of our cardinal sequence, it follows $T_{\ell,j} = \set{t^k_{\ell,j}}:{k < \kappa_\ell} = \bigcup_{i < \ell} \set{t^k_{\ell,j}}:{k < \kappa_i} \subset \bigcup_{i < \ell} T_{i,j+1}$.
Now for $i < cf(\kappa)$, the set $T_i = \bigcup_{j < \omega_1} T_{i,j}$ yields a subgraph $G(T_i)$ of finite adhesion in $G$ by Remark~\ref{rem_Trobust}. Further, the sequence $\set{T_i}:{i < cf(\kappa)}$ is increasing by (2) and continuous by (4). This completes the proof.
\end{proof}

\section{From finite adhesion to normal spanning trees}
\label{sec_mainproof}

\subsection{Normal spanning trees} In this section we complete the proof of our main result:

%\begin{theorem}
%\label{thm_colouringmain}
%Every connected graph all whose minors have countable colouring number has a normal spanning tree.
%\end{theorem}

\main*

\begin{proof}
The forwards implication follows from the fact that graphs with normal spanning trees have countable colouring number, and that the property of having a normal spanning tree is closed under taking connected minors. We prove the backwards implication by induction on $\kappa = |G|$. We may assume that $\kappa$ is uncountable, and that the statement holds for all graphs $G'$ with $|G'|<\kappa$. 
Let $ \set{G_i}:{i < \sigma}$ be a continuous chain of subgraphs of finite adhesion in $G$ from Lemma~\ref{lem_decomposition2} with $|G_i| < \kappa$ for all $i < \sigma$.

We construct by recursion on $i < \sigma$ a sequence of normal trees $\set{T_i}:{i < \sigma}$ in $G$ extending each other all with the same root, such that each tree $T_i$ contains $V(G_i)$ cofinally.\footnote{The trees $T_i$ are graph theoretic subtrees of $G$ and not to be confused with the normal partition trees from Section~\ref{sec_decomp}.} Once the recursion is complete, $T = \bigcup_{i < \sigma} T_i$ is the desired normal spanning tree for $G$.

It remains to describe the recursive construction. At limits we may simply take unions. So assume that we have already defined a normal tree $T_i$ in $G$ that cofinally contains $V(G_i)$ (the same argument applies to the base case by setting $T_{-1} = \emptyset$). In order to extend $T_i$ to a normal tree $T_{i+1}$, we rely on the following two claims, to be proved below.

\begin{claim}
\label{clm_containingcofinally2}
Every rooted tree $T_i$ containing $V(G_i)$ cofinally has finite adhesion in $G$.
\end{claim}

\begin{claim}
\label{clm_stronginduction}
Every $V(G_i)$ is a countable union of dispersed sets in $G$.  
\end{claim}

Resuming our construction, since $T_i$ is normal in $G$, the neighbourhood of every component $D$ of $G - T_i$ forms a chain in $T_i$. By Claim~\ref{clm_containingcofinally2}, this chain is finite, so there exists a maximal element $t_D \in N(D)$ in the tree order of $T_i$. Choose a neighbour $r_D$ of $t_D$ in $D$.   
By Claim~\ref{clm_stronginduction}, the set $V(G_{i+1})$ and hence $V(G_{i+1}) \cap D$ is a countable union of dispersed sets in $G$. By Jung's Theorem~\ref{thm_Jung}, there is a normal tree $T_D \subset D$ with root $r_D$ cofinally containing $V(G_{i+1}) \cap D$. Then the union of $T_i$ together with all $T_D$ and edges $t_Dr_D$ for all components $D$ of $G - T_i$ with $V(G_{i+1}) \cap D \neq \emptyset$, is a normal tree $T_{i+1}$ in $G$ containing $V(G_{i+1})$ cofinally.
\end{proof}

With a  similar proof, one obtains the perhaps interesting result that if a connected graph $G$ has a tree decomposition of finite adhesion such that all torsos have normal spanning trees, then $G$ itself has normal spanning tree.

\subsection{Proof of Claim~\ref{clm_containingcofinally2}}
For the proof, we need a simple lemma.

\begin{lemma}
\label{lem_containingcofinally}
Let $H \subset G$ be a subgraph of finite adhesion. If $T$ is a rooted tree containing $V(H)$ cofinally, then any component $D$ of $G - H$ satisfies $|D \cap T| < \infty$.
\end{lemma}

\begin{proof}
Suppose for a contradiction that some component $D$ of $G - H$ meets $T$ infinitely. 
We recursively construct disjoint $d_n-h_n$ paths $P_n$ in $\upcl{d_n}_T$ from a vertex $d_n \in D \cap T$ to a vertex $h_n \in V(H)$. Suppose that paths $P_1,\ldots,P_n \subset T$ have already been constructed. As $X = \downcl{\bigcup_{m \leq n} V(P_m)}_T$ is finite, there exists $d_{n+1} \in \p{D \cap T} \setminus X$. Since $V(H)$ is cofinal in $T$, there is a vertex $h_{n+1} \in V(H)$ above $d_{n+1}$. Let $P_{n+1}$ be the unique path in $T$ from $d_{n+1}$ to $h_{n+1}$. Since $X$ is down-closed, we have $ \upcl{d_{n+1}} \cap X = \emptyset$. Since $P_{n+1} \subset \upcl{d_{n+1}}$, it follows that $P_{n+1}$ is disjoint from all earlier paths.

However, the existence of infinitely many pairwise disjoint paths from $D$ to $H$ in $G$ contradicts that $H$ has finite adhesion in $G$.
\end{proof}

\begin{proof}[Proof of Claim~\ref{clm_containingcofinally2}]
Since $V(G_i) \subset V(T_i)$, any component $D$ of $G-{T_i}$ is contained in a unique component $\tilde{D}$ of $G - G_i$. 

Now the neighbours in $N(D) \subset T_i$ come in two types. First, those in $N(D) \cap G_i$, but these must then also belong to $N(\tilde{D})$, and there are only finitely many of these, as $G_i$ has finite adhesion in $G$. And second, neighbours $N(D) \setminus G_i$, but they must then belong to $T_i \cap \tilde{D}$, and there are only finitely many of those by Lemma~\ref{lem_containingcofinally}.
\end{proof}

\subsection{Proof of Claim~\ref{clm_stronginduction}} By our main induction, all $G_i$ have normal spanning trees. However, we need to wrest the stronger assertion from our induction assumption that each $V(G_i)$ is contained in a normal (not necessarily spanning) tree of $G$, which by Jung's Theorem~\ref{thm_Jung} is precisely the assertion of Claim~\ref{clm_stronginduction}. For the proof, we need two definitions and lemma.

\begin{defn}[Dominated torsos] 
 \label{sec_weaktorso}%The \emph{torso} of $G_i$ is the supergraph of $G_i$ obtained by making all adhesion sets complete. 
For an adhesion set $A$ in $G_i$, let $\script{D}_A$ be the set of components of $G-G_i$ with $N(D) = A$. The \emph{dominated torso} $\hat{G}_i$ of $G_i$ is the minor $\hat{G}_i \preceq G$, where
%with $G_i \subset \hat{G}_i$ 
%defined as follows:
\begin{enumerate}[label=(T\arabic*)]
	\item\label{wt1} for adhesion sets $A$ with $\script{D}_A$ finite, we contract each $D \in \script{D}_A$ to a single vertex $v_D$, and
	\item\label{wt2} for adhesion sets $A$ with $\script{D}_A$ infinite, we choose a partition $\script{D}_A = \bigsqcup_{a \in A} \script{D}_a$ and contract %all edges of 
	each connected graph $G[\Set{a} \cup \bigcup \script{D}_a]$ to a single vertex, identified with $a$.
\end{enumerate} 
\end{defn}

With these identifications we naturally have $G_i \subset \hat{G}_i$. Every adhesion set $A \subset G_i$ of the second type induces a clique in $\hat{G}_i$, and every adhesion set of the first type has a dominating vertex $v_D$.

\begin{defn}[Canonical projection] Let $P \subset G$ be a path with end vertices in $G_i$ or a ray which meets $G_i$ again and again. A \emph{$G_i$-fragment} $Q$ of $P$ is a connected component $\mathring{Q}$ of $P - G_i$ together with its two edges from $\mathring{Q}$ to $G_i$ and their ends (so every $P$-fragment is a $G_i$-path). The \emph{canonical projection of $P$ to $\hat{G_i}$} is the tour $\hat{P}\subset \hat{G}_i$ obtained by replacing every $G_i$-fragment $Q=xPy$ of $P$ by 
\begin{itemize}
\item the path $x v_D y$, if $\mathring{Q}$ is contained in a component $D$ of $G - G_i$ as in \ref{wt1}, or by 
\item the edge $xy$, if $\mathring{Q}$ is contained in a component $D$ of $G - G_i$ as in \ref{wt2}.
\end{itemize}
\end{defn}

\begin{lemma}
\label{lem_projectionproperties}
Let $P \subset G$ be a path with end vertices in $G_i$ or a ray meeting $G_i$ again and again.
\begin{enumerate}
\item	The canonical projection $\hat{P}$ of $P$ to $\hat{G}_i$ satisfies $V(P) \cap G_i = V(\hat{P}) \cap G_i$.
\item The canonical projection $\hat{P}$ of $P$ to $\hat{G}_i$ is a locally finite tour.
\item Let  $X \subset V(G_i)$ be a finite set of vertices, and $v,w \in V(G_i) \setminus X$. If $X$ separates $v$ from $w$ in $\hat{G_i}$, then it also separates $v$ from $w$ in $G$.
\end{enumerate}
\end{lemma}

\begin{proof}
Assertion (1) is immediate. Assertion (2) holds as the only vertices in $\hat{P}$ used more than once are of the form $v_D$ as in \ref{wt1}, but these have finite degree in $\hat{G}_i$. Assertion (3) follows, since if $P$ is a $v-w$ path in $G$ avoiding $X$, then $\hat{P}$ is a $v-w$ tour by (2) avoiding $X$ by (1).
\end{proof}

\begin{proof}[Proof of Claim~\ref{clm_stronginduction}]
Since $G_i$ has finite adhesion in $G$, %there are only $|G_i|^{<\omega} = |G_i|$ many adhesion sets of $G_i$. Hence, 
the number of vertices $v_D$ in $\hat{G}_i$ as in \ref{wt1} is at most $|G_i|$. Therefore, we have $|\hat{G}_i| = |G_i| < \kappa$, and so inductively, every dominated torso $\hat{G}_i$ has a normal spanning tree. By Jung's Theorem~\ref{thm_Jung}, the set of vertices $V(G_{i})$ in the graph $\hat{G}_{i}$ is a countable union of sets $\set{U_n}:{n \in \N}$ which are dispersed in $\hat{G}_{i}$.

It remains to show that each $U_n$ is also dispersed in $G$. Consider an arbitrary ray $R$ in $G$. If $R$ is eventually contained in a component $D$ of $G - G_{i}$, then $N(D)$ separates $U_n$ from a tail of $R$. Otherwise, $R$ meets $G_{i}$ again and again. Let $\hat{R}$ be the canonical projection of $R$ to $\hat{G}_i$, which is an infinite, locally finite tour by Lemma~\ref{lem_projectionproperties}(1)\&(2). Since $U_n$ is dispersed in $\hat{G}_{i}$, there is a finite set of vertices $\hat{X} \subset V(\hat{G}_{i})$ separating $\hat{R}$ from $U_n$ in $\hat{G}_{i}$. Let $X$ denote the finite subset of $V(G_i)$ where we replace every vertex in $\hat{X}$ of the form $v_{D}$ as in \ref{wt1} by all vertices in $N(D)$. Then $X \subset V(G_{i})$ separates $\hat{R}$ from $U_n$ in $\hat{G}_{i}$. By Lemma~\ref{lem_projectionproperties}(3), the set $X$ then also separates $R$ from $U_n$ in $G$. Hence, $U_n$ is dispersed in $G$ as desired.
\end{proof}

\section{A forbidden minor characterisation for normal spanning trees}
\label{sec_forb}

The main result of \cite{bowler2015colouring} is a forbidden subgraph characterisation for the property of ``having colouring number $\leq \mu$''. The forbidden subgraphs for the case $\mu = \aleph_0$ are the following:

\begin{defn}
\label{def_aleph0obstruction}
(1) A $(\lambda,\lambda^+)$\emph{-graph} for some infinite cardinal $\lambda$ is a bipartite graph $(A,B)$ such that $\cardinality{A}=\lambda$, $\cardinality{B}=\lambda^+$, and every vertex in $B$ has infinite degree. % (and without loss of generality, every vertex in $A$ has degree $\lambda^+$, \cite[Lemma~2.4]{bowler2015colouring}).
	
(2) A $(\kappa,S)$\emph{-graph} for some regular uncountable cardinal $\kappa$ and some stationary set $S \subset \kappa$ of cofinality $\omega$ ordinals 
is a graph with vertex set $V(G) = \kappa$ such that $N(s) \cap \set{v \in \kappa}:{v < s}$ is countable with supremum $s$ for all $s \in S$.
\end{defn}

\maintwo*

\begin{proof}
It is proved in \cite{bowler2015colouring} that a graph $G$ has countable colouring number if and only if $G$ does not contain a $(\lambda,\lambda^+)$-graph nor a  $(\kappa,S)$-graph as a subgraph. Hence, this forbidden subgraph characterisation for countable colouring number translates, via Theorem~\ref{thm_Halin's_conj}, to a forbidden minor characterisation for normal spanning trees. 

It remains to argue that for the backwards implication it suffices to exclude these obstructions as minors with countable branch sets. Towards this end, recall that if a graph contains an uncountable clique minor, then it also contains such a minor with countable branch sets: Indeed, if there is an uncountable clique minor, then one also gets a subdivision of an uncountable clique by a result of Jung \cite{jung1967zusammenzuge}, say on branch vertices $\set{v_i}:{i < \omega_1}$. Contracting each $v_i$ together with all subdivided paths leading from $v_i$ to earlier vertices $v_j$ for $j < i$ yields the desired minor with countable branch sets. Thus, after excluding uncountable clique minors with countable branch sets, Lemma~\ref{lem_noOmega_1chains} still implies that $\GT \preceq G$ has all branch sets countable. In Lemma~\ref{lem_closureTrobust}, two further minors of $G$ and $\GT$ respectively are considered: 
\begin{itemize}
	\item A barricade minor of $G$; in the  proof, a number of components $D$ where contracted. However, for each component is suffices to contract a countable subtree $T_D \subset D$ so that the resulting vertex has infinite degree in the barricade.
	\item An Aronszajn tree minor of $\GT$, obtained by contracting the countable subset $T_0$ of $\GT$.
	\end{itemize}

Next, in the proof of Lemma~\ref{lem_decomposition} for regular $\kappa$, we chose a well-order of $V(\GT)$ witnessing that $\GT$ has countable colouring number. By \cite{bowler2015colouring}, this  requires that there are no $(\lambda,\lambda^+)$-graphs or $(\mu,S)$-graphs contained in $\GT$ as subgraphs, which is again fine as $\GT \preceq G$ has all branch sets countable.

Finally, we consider minors $\hat{G}_i \preceq G$ % in the proof of Theorem~\ref{thm_Halin's_conj} 
in Section~\ref{sec_mainproof}. Since $G_i \subset G$ has finite adhesion, these minors can be realised using finite branch sets; so any forbidden minor occurring with countable branch sets in $\hat{G}_i$ also occurs as minor with countable branch sets in $G$.
\end{proof}

\section{Diestel's normal spanning tree criterion}
\label{sec_fat}

The original proof of Diestel's sufficient condition from \cite{diestel2016simple}, that graphs without fat $TK_{\aleph_0}$ have normal spanning trees, relied on the incorrect forbidden minor characterisation from \cite{DiestelLeaderNST}. This section contains a proof of Diestel's criterion based on the obstructions from Theorem~\ref{thm_forbiddenminorsIntro}.

A graph $G$ is an $IX$ (an \emph{inflated} $X$) if $X \preceq G$ such that $V(G)=\bigcup V_x$. We say an $IX$ graph $G$ is \emph{countably inflated} if $V_x \subset V(G)$ is countable for all $x \in X$.

\begin{theorem}
All graphs not containing a fat $TK_{\aleph_0}$ as subgraph have normal spanning trees.
\end{theorem}

\begin{proof}
By Theorem~\ref{thm_forbiddenminorsIntro} it suffices to show: \emph{Every countably inflated version $IX$ of a $(\lambda,\lambda^+)$-graph or a  $(\kappa,S)$-graph $X$ contains a fat $TK^{\aleph_0}$.}

We first deal with $(\lambda,\lambda^+)$-graphs. Let $(A,B)$ be the bipartition for $X$ with $|A|=\lambda$ and $|B| = \lambda^+$, and suppose that $H$ is a countably inflated $IX$. Construct, inductively, an infinite set $\Set{a_1,a_2,a_3, \ldots} \subset A$ and a nested sequence $B_1 \supset B_2 \supset \cdots$ of $\lambda^+$ sized subsets of $B$ such that for each $a_n$ there is a vertex $h_n \in V_{a_n}$ which sends edges in $H$ to all branch sets of $b \in B_n$. Suppose the construction has proceeded to step $n$. Fix a new vertex $a_{n+1}$ whose $X$-neighbourhood $B'_{n+1}$ in $B_n$ has size $\lambda^+$. To see that this is possible, note that if there is no vertex $a \in A \setminus \Set{a_1,a_2,\ldots,a_n}$ as claimed, then each vertex $a \in A \setminus \Set{a_1,a_2,\ldots,a_n}$ has as most $\lambda$ many neighbours in $B_n$. As $A \setminus \Set{a_1,a_2,\ldots,a_n}$ has size $\lambda$, this means that $B_{n} \setminus N(A \setminus \Set{a_1,a_2,\ldots,a_n}) \neq \emptyset$. But every vertex in this set has all its neighbours in $\Set{a_1,a_2,\ldots,a_n}$, and thus has finite degree, a contradiction. 
Since $V_{a_{n+1}}$ is countable and $\lambda^+$ is regular uncountable, there is a vertex $h_{n+1} \in  V_{a_{n+1}}$ which receives edges from $\lambda^+$ branch sets of distinct vertices in $B'_{n+1}$, and we call this set of vertices $B_{n+1}$. This completes the inductive construction. 

Since we may enumerate the edges needed for a fat $TK^{\aleph_0}$ in order type $\omega_1$ and $\lambda^+ \geq \omega_1$,  it is then routine to construct a fat $TK^{\aleph_0}$ with branch vertices $h_1,h_2,h_3,\ldots$ in $H$.

Next, we deal with $(\kappa,S)$-graphs $X$. Suppose that $H$ is a countably inflated $IX$. Let us enumerate the down-neighbours of $s \in S$ in $X$ by $v^s_1,v^s_2,\ldots$. We construct, inductively, an infinite set $\Set{v_1,v_2,v_3, \ldots} \subset X$ and a nested sequence $S_1 \supset S_2 \supset \cdots$ of stationary subsets of $S$ such that for each $v_n$ there is a vertex $h_n \in V_{v_n}$ which sends edges to all branch sets of $s \in S_n$, and each $s \in S_n$ satisfies $v^s_i=v_i$ for all $i \leq n$. By applying Fodor's Lemma~\ref{lem_stationary}(2), find a vertex $v_{n+1}$ and a stationary subset $S'_{n+1} \subset S_n$ such that $v^s_i = v_i$ for all $i \leq n+1$ and all $s \in S'_{n+1}$. Since $V_{v_{n+1}}$ is countable and $S'_{n+1}$ is stationary, by Lemma~\ref{lem_stationary}(1) there is a vertex $h_{n+1} \in  V_{a_{n+1}}$ which receives edges from stationary many distinct branch sets of vertices in $S'_{n+1}$, and we call this set of vertices $S_{n+1}$. This completes the inductive construction. 

Once again, it is then routine to construct a fat $TK^{\aleph_0}$ with branch vertices some infinite subset of $\Set{h_1,h_2,h_3,\ldots}$ in $H$.
\end{proof}

\section{Further problems on normal spanning trees and forbidden minors}

\begin{prob}
Is there a list of forbidden minors for the property of having a normal spanning tree consisting of all $(\lambda,\lambda^+)$-graphs, and a list of $T$-graphs? 
\end{prob}

\begin{prob}
\label{prob2}
Is it consistent with the axioms of set theory ZFC that it suffices in Theorem~\ref{thm_forbiddenminorsIntro} to forbid minors of cardinality $\aleph_1$?
\end{prob}

An earlier result of the author \cite[Theorem~5.1]{pitz2020new} shows that the opposite assertion is consistent with ZFC as well, that one needs to forbid minors of arbitrarily large cardinality in Theorem~\ref{thm_forbiddenminorsIntro}.

\begin{prob}
\label{prob3}
Is it true that a graph has a normal spanning tree if and only if its vertex set is the countable union of fat $TK^{\aleph_0}$-dispersed sets? Here, a set  of vertices $U$ is \emph{fat $TK^{\aleph_0}$-dispersed} if the branch vertices of every fat $TK^{\aleph_0}$ can be separated from $U$ by a finite set of vertices.
\end{prob}

An affirmative result to  the last problem would both generalise Jung's Theorem~\ref{thm_Jung} as well as a result by the author \cite{pitz2020}.

\bigskip
\textbf{Addendum to the ArXiv version:} An affirmative answer to Problem~\ref{prob2} follows from Theorem~\ref{thm_Halin's_conj} and a result of Komj\'ath \cite{komjath1987colouring} that it is consistent that every graph of uncountable colouring number number contains a subgraph of size and colouring number $\aleph_1$. An affirmative answer to Problem~\ref{prob3} has been given by the author in \cite{pitz2021quickly}.

\bibliographystyle{plain}
\bibliography{Pitz_ProofHalinsConjecture}

\begin{thebibliography}{10}

\bibitem{bowler2015colouring}
N.~Bowler, J.~Carmesin, P.~Komj{\'a}th, and Chr. Reiher.
\newblock The colouring number of infinite graphs.
\newblock {\em Combinatorica}, 39:1225--1235, 2019.

\bibitem{brochet1994normal}
J.-M. Brochet and R.~Diestel.
\newblock Normal tree orders for infinite graphs.
\newblock {\em Transactions of the American Mathematical Society},
  345(2):871--895, 1994.

\bibitem{jancarl}
C.~B{\"u}rger and J.~Kurkofka.
\newblock Duality theorems for stars and combs {I}: {A}rbitrary stars and
  combs.
\newblock {\em Journal of Graph Theory}, 99(4):525--554, 2022.

\bibitem{Bible}
R.~Diestel.
\newblock {\em {Graph Theory}}.
\newblock Springer, 5th edition, 2015.

\bibitem{diestel2016simple}
R.~Diestel.
\newblock A simple existence criterion for normal spanning trees.
\newblock {\em The Electronic Journal of Combinatorics}, 2016.
\newblock P2:33.

\bibitem{DiestelLeaderNST}
R.~Diestel and I.~Leader.
\newblock Normal spanning trees, {A}ronszajn trees and excluded minors.
\newblock {\em Journal of the London Mathematical Society}, 63:16--32, 2001.

\bibitem{EGJKP19}
J.~Erde, P.~Gollin, A.~Jo\'{o}, P.~Knappe, and M.~Pitz.
\newblock A {C}antor-{B}ernstein-type theorem for spanning trees in infinite
  graphs.
\newblock {\em Journal of Combinatorial Theory, Series B}, 149:16--22, 2021.

\bibitem{halin2000miscellaneous}
R.~Halin.
\newblock Miscellaneous problems on infinite graphs.
\newblock {\em Journal of Graph Theory}, 35(2):128--151, 2000.

\bibitem{jung1967zusammenzuge}
H.A. Jung.
\newblock Zusammenz{\"u}ge und {U}nterteilungen von {G}raphen.
\newblock {\em Mathematische Nachrichten}, 35(5-6):241--267, 1967.

\bibitem{jung1969wurzelbaume}
H.A. Jung.
\newblock Wurzelb{\"a}ume und unendliche {W}ege in {G}raphen.
\newblock {\em Mathematische Nachrichten}, 41(1-3):1--22, 1969.

\bibitem{komjath1987colouring}
P.~Komj{\'a}th.
\newblock The colouring number.
\newblock In {\em Proc. London Math. Soc.(3)}, volume~54, pages 1--14, 1987.

\bibitem{Kunen}
K.~Kunen.
\newblock Set theory, {V}olume 34 of {S}tudies in {L}ogic, {C}ollege
  {P}ublications, {L}ondon, 2011.

\bibitem{pitz2020}
M.~Pitz.
\newblock A unified existence theorem for normal spanning trees.
\newblock {\em Journal of Combinatorial Theory, Series B}, 145:466--469, 2020.

\bibitem{pitz2020new}
M.~Pitz.
\newblock A new obstruction for normal spanning trees.
\newblock {\em Bulletin of the London Mathematical Society}, 53(4):1220--1227,
  2021.

\bibitem{pitz2021quickly}
M.~Pitz.
\newblock Quickly proving {D}iestel's normal spanning tree criterion.
\newblock {\em The Electronic Journal of Combinatorics}, 28(3), 2021.
\newblock P3:59.

\end{thebibliography}

\end{document}